\documentclass[12pt]{amsart}
\pdfoutput=1
\usepackage{amsmath, amssymb}
\usepackage{amsfonts}
\usepackage{mathrsfs}
\usepackage[arrow,matrix,curve,cmtip,ps]{xy}
\usepackage{graphicx}
\usepackage{amsthm}
\usepackage{float}
\usepackage{amsthm}
\usepackage[utf8]{inputenc}
\usepackage[T1]{fontenc}
\usepackage{mathtools}
\usepackage{caption,subcaption}   

\allowdisplaybreaks

\newtheorem{theorem}{Theorem}[section]
\newtheorem{lemma}[theorem]{Lemma}
\newtheorem{proposition}[theorem]{Proposition}
\newtheorem{corollary}[theorem]{Corollary}

\newtheorem*{theorem*}{Theorem}
\theoremstyle{remark}
\newtheorem{remark}[theorem]{Remark}

\newtheorem{question}{Question}

\numberwithin{equation}{section}

\begin{document}
\title{On curves and polygons with the equiangular chord property}

\author{Tarik Aougab, Xidian Sun, Serge Tabachnikov, Yuwen Wang}

\address{Department of Mathematics \\ Yale University \\ 10 Hillhouse Avenue, New Haven, CT 06510 \\ USA}
\email{tarik.aougab@yale.edu}

\address{Department of Mathematics \\ Wabash College \\ 301 W. Wabash Avenue, Crawfordsville, IN 47933 \\ USA}
\email{xsun15@wabash.edu}

\address{Department of Mathematics \\ Penn State University \\ University Park, State College 16802 \\ USA\\
and ICERM\\ Brown University\\ Box 1995, Providence, RI 02912\\ USA}
\email{tabachni@math.psu.edu}

\address{Department of Mathematics and Statistics \\ Swarthmore College \\ 500 College Avenue, Swarthmore, PA 19081 \\ USA}
\email{ywang3@swarthmore.edu}

\date{\today}


\begin{abstract}
Let $C$ be a smooth, convex curve on either the sphere $\mathbb{S}^{2}$, the hyperbolic plane $\mathbb{H}^{2}$ or the Euclidean plane $\mathbb{E}^{2}$, with the following property: there exists $\alpha$, and parameterizations $x(t), y(t)$ of $C$ such that for each $t$, the angle between the chord connecting $x(t)$ to $y(t)$ and $C$ is $\alpha$ at both ends. 

Assuming that $C$ is not a circle,  E. Gutkin completely characterized the angles $\alpha$ for which such a curve exists in the Euclidean case. We study the infinitesimal version of this problem in the context of the other two constant curvature geometries, and in particular we provide a complete characterization of the angles $\alpha$ for which there exists a non-trivial infinitesimal deformation of a circle through such curves with corresponding angle $\alpha$. We also consider a discrete version of this property for Euclidean polygons, and in this case we give a complete description of all non-trivial solutions.

\end{abstract}

\maketitle

\hfill{ \it To the memory of Eugene Gutkin}

\section{Introduction}
 Given a smooth, convex oriented closed curve $C$ in the Euclidean plane $\mathbb{E}^{2}$ and $x,y\in C$, $x\neq y$, let $|xy|$ denote the oriented chord connecting $x$ to $y$.  Motivated by his study of mathematical billiards, E. Gutkin asked the following question \cite{Gu1}: 

\begin{question}   Assume the existence of  parameterizations $x(t), y(t)$ of $C$ such that for each $t$,
\begin{enumerate}
\item $x'(t),y'(t) >0$;
\item $x(t) \neq y(t)$; 
\item there exists $\alpha \in (0,\pi]$ such that both angles between $C$ and $|x(t)y(t)|$ is $\alpha$. 
\end{enumerate}

Then if $C$ is not a circle, what are all possible values of $\alpha$?

\end{question}

Gutkin provides a complete answer to Question $1$ by establishing the following necessary and sufficient condition for $\alpha$: there
exists an integer $k \geq 2$ such that 
\begin{equation} \label{restr1}
k \tan \alpha=\tan{(k\alpha)},
\end{equation}
see \cite{Gu1},\cite{Gutkin},\cite{Tabachnikov1}. In particular, only a countable number of values of the angle $\alpha$ are possible.

In terms of billiards, the billiard ball map on the interior of $C$ has a horizontal invariant circle given by the condition that the angle made by the trajectories with the boundary of the table is equal to $\alpha$. This statement can also be interpreted in terms of capillary floating with zero gravity in neutral equilibrium, see \cite{RFinn1,RFinn2}.

We call a curve satisfying this equiangular chord property a \textit{Gutkin curve}; we will refer to the corresponding angle $\alpha$ as the {\it contact angle}. 

We generalize Gutkin's theorem in two directions: to curves in the standard $2$-sphere $\mathbb{S}^2$ and the hyperbolic plane $\mathbb{H}^2$, and to polygons in $\mathbb{E}^2$ via a discretized version of Question $1$. For $\mathbb{S}^{2}$ and $\mathbb{H}^{2}$, we consider the following infinitesimal version of Gutkin's question:

\begin{question} In either $\mathbb{H}^{2}$ or $\mathbb{S}^{2}$, for which angles $\alpha$ are there non-trivial infinitesimal deformations of a radius $R$ circle through Gutkin curves with contact angle $\alpha$?

\end{question}

Here, a \textit{non-trivial deformation} of a circle is a deformation that does not correspond to a circle solution (of a different radius). 

Our first result yields an answer to Question $2$:

\begin{theorem} \label{Conscurv}
Assume that a circle of radius $R$ in $\mathbb{S}^2$ or in $\mathbb{H}^2$ admits a non-trivial infinitesimal deformation
through Gutkin curves with contact angle $\alpha$. Define  angles $c$ via
$$
\cot c = \cos R \cot \alpha,
$$
in the spherical case, and 
$$
\cot c = \cosh R \cot \alpha,
$$
in the hyperbolic case. 
Then    there exists $k\in \mathbb{N}, k\geq 2$, such that the following
equation holds:
$$
k \tan c = \tan kc.
$$
\end{theorem} 

Thus, as in the Euclidean case, only a countable number of values of the contact angle $\alpha$ are possible for a given radius $R$. 

Note that, in the Euclidean plane, Gutkin curves with contact angle $\alpha= \frac{\pi}{2}$ are precisely the curves of constant width; the same holds in the spherical and hyperbolic settings;  see \cite{Le} for curves of constant width in non-Euclidean geometries.

In section \ref{sect:polygon}, we consider the following analog of Gutkin's theorem for
polygons in $\mathbb{E}^2$. Let P be a convex $n$-gon with vertices $\left\{v_{0},...,v_{n-1}\right\}$ in their 
cyclic order. For $k\in \mathbb{N}, 2\leq k\leq n/2$, a $k$-diagonal is a straight line segment connecting vertices of $P$ whose indices differ by $k$, modulo $n$.  Then $P$ is a \textit{non-trivial Gutkin (n,k)-gon} if $P$ is not regular and there exists $\alpha$ such that for any $k$-diagonal $D$, both contact angles between $D$ and $P$ equal $\alpha$ (see Figure \ref{twoex} for examples). That is, for each $i$,  
\[ \angle v_{i+1}v_{i}v_{i+k} =\angle v_{i+k-1}v_{i+k}v_{i} =\alpha, \]
where $\angle v_{i+1}v_{i}v_{i+k}$ denotes the angle between the edge $|v_{i+1}v_{i}|$ and the $k$-diagonal $|v_{i}v_{i+k}|$. 

\begin{figure}[H]
\centering
\includegraphics[height=2in]{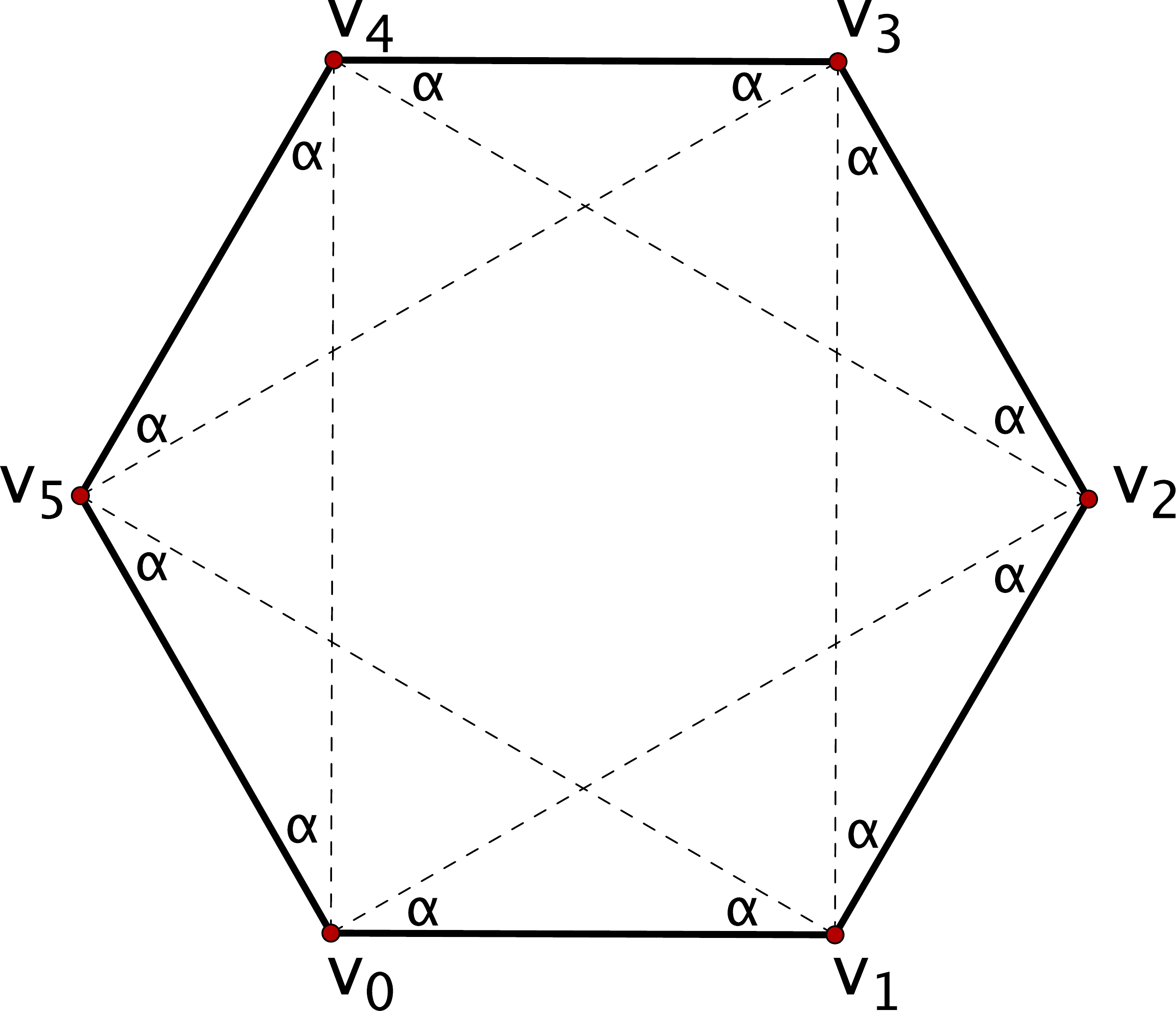}
\quad
\includegraphics[height=2in]{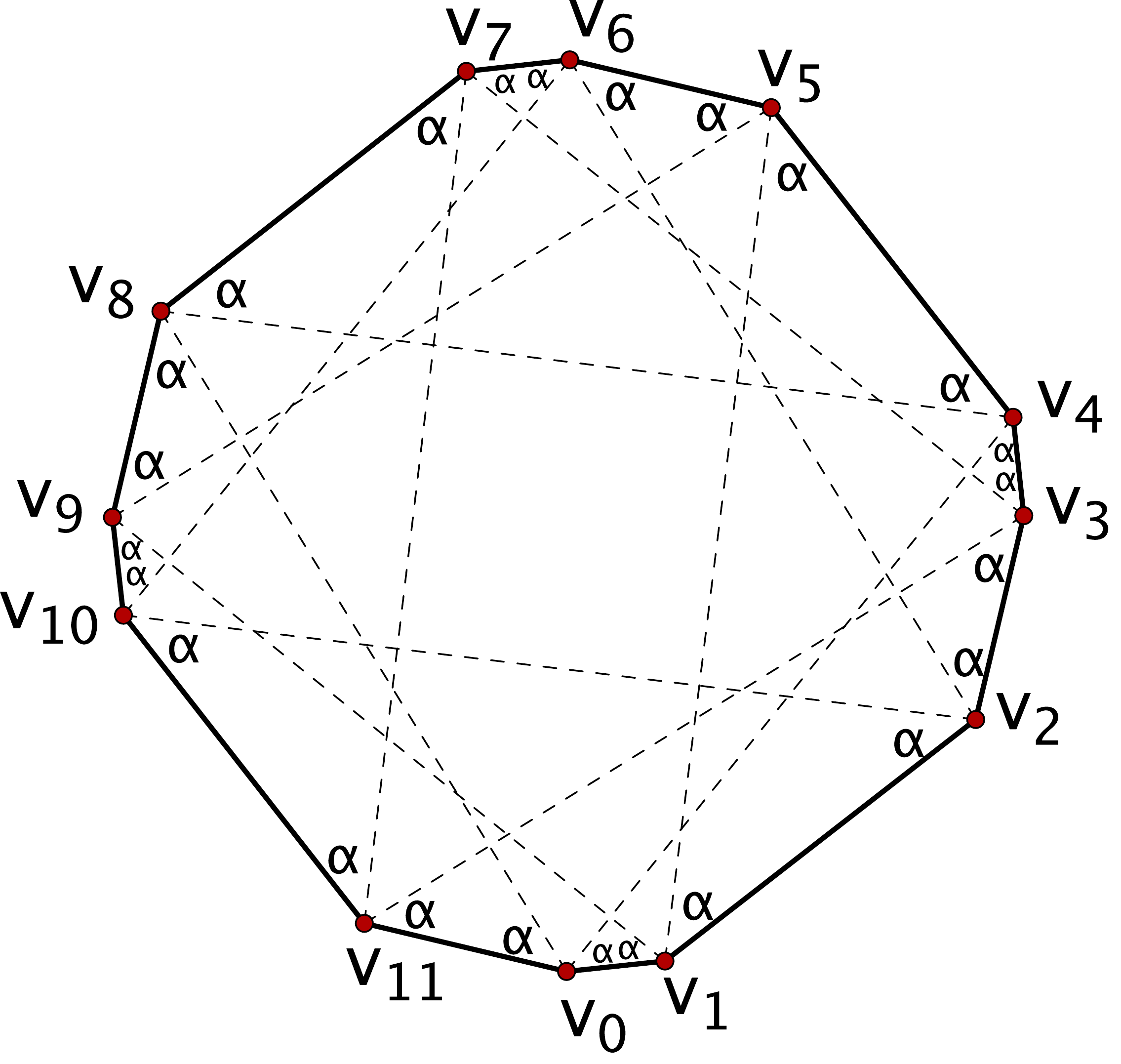}
\caption{Gutkin $(6,2)$-gon and $(12,4)$-gon}
\label{twoex}
\end{figure}

Our second result is a complete characterization of the pairs $(n,k)$ for which a non-trivial Gutkin $(n,k)$-gon exists:

\begin{theorem} \label{mainpol}
A non-trivial Gutkin $(n,k)$-gon in the Euclidean plane exists if and only if $n$ and $k-1$ are not coprime. 
\end{theorem}

Interestingly, the main ingredient of our proof is the Diophantine equation
$$
\tan \left(\frac{kr \pi}{n}\right) \tan \left(\frac{\pi}{n}\right) = \tan \left(\frac{k \pi}{n}\right) \tan \left(\frac{r \pi}{n}\right)
$$
which is a discrete version of (\ref{restr1}). This equation also appeared in \cite{Tabachnikov2}, and it was solved in \cite{connelly}.

\section{A proof of Gutkin's theorem in $E^2$}\label{section2}
Although the existing proofs of Gutkin's theorem in $\mathbb{E}^2$ \cite{Gu1,Gutkin,Tabachnikov1} are very clear and simple, our goal in this paper is to study the situation in $\mathbb{S}^2$ and $\mathbb{H}^2$. Therefore, in this section we reprove (the necessary part of) Gutkin's theorem using methods which can be applied to the other constant curvature settings. This proof is motivated by the study of integrable  billiards by M. Bialy  \cite{Bialy1,Bialy2}.

\begin{figure}[H]
\centering
\includegraphics[height=2.2in]{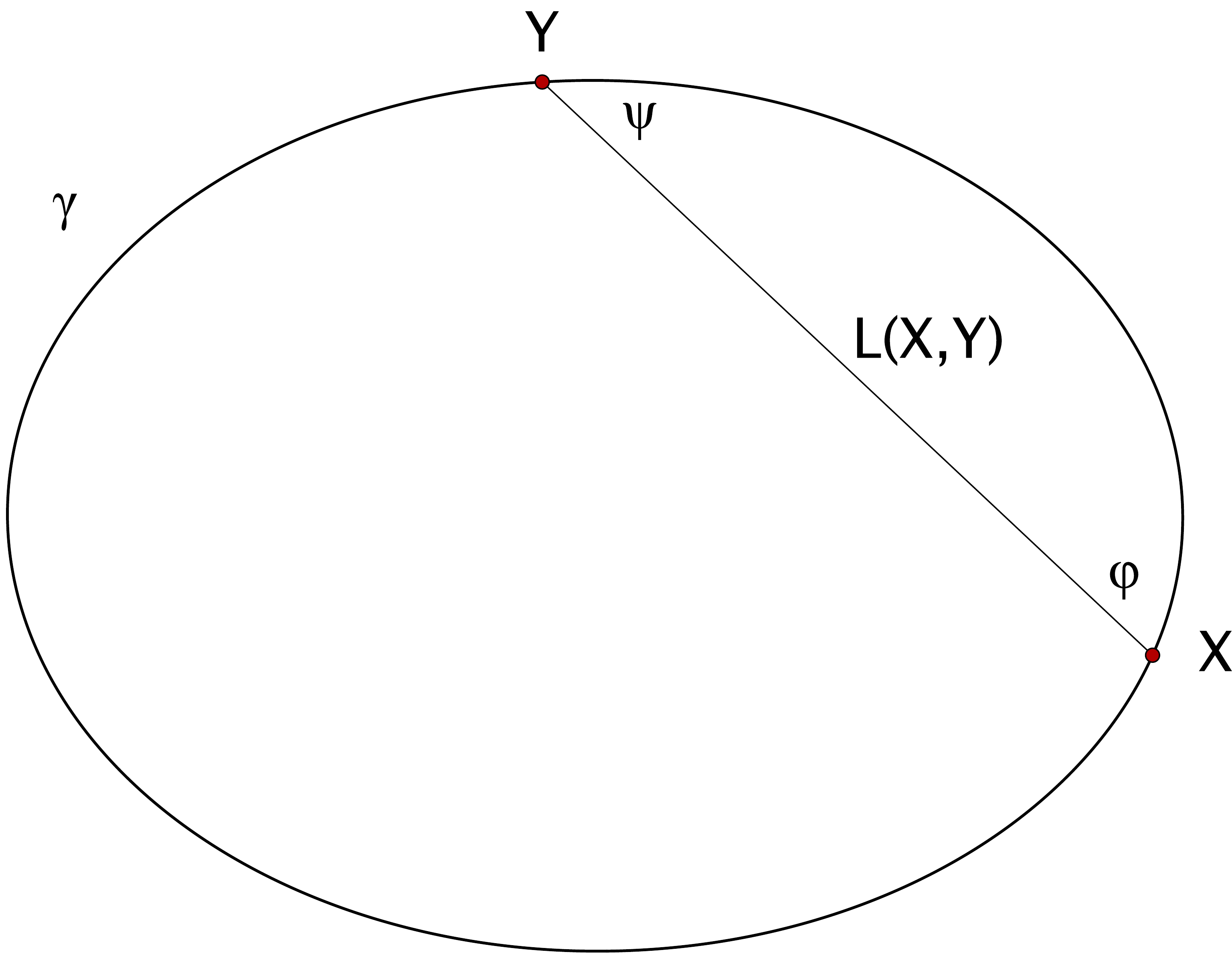}
\caption{\label{fig:continuous} $\gamma$ with chord $xy$}
\end{figure}

Let $\gamma$ be a smooth, convex and closed curve parameterized by arc length as shown in Figure \ref{fig:continuous}, $x$ and $y$ be points on $\gamma$ ($x$ and $y$ are arc length parameters), $\phi$ and $\psi$ the angles made by the chord  $xy$ with $\gamma$, and $L$ the length of the chord, the generating function of the billiard ball map.
 We have that 
 $$
 L_x = - \cos {\phi},\quad L_y = \cos{\psi},
 $$
\begin{equation}\label{eq:E2Lderiv}
L_{xy}=\frac{\sin{\phi}\sin{\psi}}{L}, \quad L_{xx} = \frac{\sin^2 \phi}{L} -
\kappa(x) \sin \phi, \quad L_{yy} = \frac{\sin^2 \psi}{L} - \kappa(y) \sin \psi,
\end{equation}
where $\kappa$ is the curvature of the curve, and subscripts denote partial differentiation, see, e.g., \cite{Bialy1}.

We interpret $L(x,y)$ as a function on the torus $\gamma \times \gamma$; then there exists a curve $s$ on this torus where both angles, $\phi$ and $\psi$, have the same constant value  $\alpha$.

We seek a parameter $t$ on $\gamma$ so that the values of this parameter at points $x$ and $y$ differ by a constant: $t(y)=t(x)+2c$. 
Denote $d/dt$ by prime.

\begin{proposition} \label{turn}
The parameter $t$ is determined by the condition $x' = a/\kappa(x)$, where
$a$ is a constant.
\end{proposition}
\begin{proof}
Since $\alpha$ is constant as a function of $t$, 
\begin{equation}\label{eq:Lderiv_simple}
0=L_{xt} = L_{xx} x' + L_{xy} y',\quad 
0=L_{yt} = L_{xy} x'+ L_{yy} y'.
\end{equation}
This implies that $L_{xx} L_{yy} = L_{xy}^2$ along our curve, and substituting
from equation (\ref{eq:E2Lderiv}), we have 
\begin{equation}\label{nice}
\frac{\sin\alpha}{\kappa(x)}+\frac{\sin\alpha}{\kappa(y)} = L.
\end{equation}
We  compute $y'/x'$ from equations (\ref{eq:E2Lderiv})--(\ref{nice}):
$$\frac{y'}{x'}=-\frac{L_{xy}}{L_{yy}}= \frac{\sin\alpha}{\kappa(y) L-\sin\alpha}=
\frac{\kappa(x)}{\kappa(y)},$$
which implies the claim.
\end{proof}

Since the curvature is the rate of turning of the direction of the curve, Proposition \ref{turn} defines (up to a multiplicative coefficient) the angular parameter along the curve. Note that $0\leq x \leq L(\gamma)$, and $0 \leq t \leq T$, where $T$ is the upper bound of $t$ and $L(\gamma)$ is the length of $\gamma$. It follows that 
$$T=\int_0^T \mathrm{d}t=\frac{1}{a} \int_0^{L(\gamma)} \kappa(x) \mathrm{d}x.$$
Choose $a=1$ to make $T=2\pi$, which agrees with the angle. Then $c=\alpha$.

In view of Proposition \ref{turn}, we set 
$$ f_1:=f(t-\alpha)=\frac{\sin \alpha}{\kappa(x)},\quad  f_2:=f(t+\alpha)=\frac{\sin \alpha}{\kappa(y)}.$$
 From equation (\ref{nice})  we have:
 $$ L=\frac{\kappa(x)\sin \alpha+\kappa(y)\sin \alpha}{\kappa(x)\kappa(y)}=f_1+f_2.$$
 It follows that $L'={f_1'}+{f_2'}$. By the chain rule, we have that
$$L_x {x'}+L_y {y'}={\cot{\alpha}}\ (f_1-f_2)={f'_1}+{f'_2},$$
and therefore
\begin{equation}\label{eq:E2f}
f'(t+\alpha)+f'(t-\alpha)= \cot{\alpha} \left( f(t+\alpha)-f(t-\alpha)\right).
\end{equation}

Since $f(t)$ is a function with period $2\pi$,  using the Fourier expansion, we obtain $f(t)=\Sigma b_k e^{ikt}$, where $b_k \in \mathbb{C}$, and $b_{-k}=\overline{b_k}$. Thus
$$f(t\pm \alpha)=\Sigma b_k e^{\pm ik\alpha} e^{ikt}, \quad f'(t\pm \alpha)=\Sigma b_k ik e^{\pm ik\alpha} e^{ikt}.$$ 
Let $LHS$ be the left hand side of equation (\ref{eq:E2f}) and $RHS$ be the right hand side. It follows that 
$$LHS=\Sigma b_k ik \big(e^{ik\alpha}+e^{-ik\alpha}\big) e^{ikt},\ RHS= \cot{\alpha} \Sigma b_k
\big( e^{ik\alpha}-e^{-ik\alpha}\big) e^{ikt}.$$ 
Equating both sides, we have that 
$$b_k (k \cos{k\alpha}-\cot{\alpha} \sin{k\alpha})=0.$$
For $k=1$, this automatically holds, and if $b_k\neq 0$ for some $k\geq 2$ then
\begin{equation*}
k \tan{\alpha}= \tan{k\alpha}.
\end{equation*}

If the curve is a circle then $f(t)$ is constant and all $b_k=0$, and if the curve is not a circle then $b_k\neq 0$ for some $k\geq 1$. It remains to show that $b_1=0$.

Recall that $x$ is arc length and $t$ is the angular parameters on the curve $\gamma$. Then $ \gamma_x=(\cos t, \sin t),\ dt/dx=\kappa.$ Therefore 
$$ \gamma_t=\frac{1}{\kappa} (\cos t, \sin t)\ \ {\rm and}\ \
\int_0^{2\pi} \gamma_t\ dt = 0.$$ 
Hence 
$$ \int_0^{2\pi} \frac{\cos t }{\kappa}\ dt = \int_0^{2\pi} \frac{\sin t }{\kappa}\ dt = 0,$$
that is, the  function $f$ is $L^2$-orthogonal to the first harmonics. Hence 
 $f$ has no first harmonics in the Fourier expansion, that is, $b_1=0$.

\section{Infinitesimal Analogs of Gutkin's theorem in $\mathbb{S}^2$ and $\mathbb{H}^2$}\label{section3}

We prove Theorem \ref{Conscurv} in detail for $\mathbb{S}^{2}$. The hyperbolic case being analogous, we only indicate the necessary changes. 

Let $\gamma$ be a Gutkin curve and, as before, let $x$ and $y$ be arc length parameters. Then  $\phi$ and $\psi$ should have constant value, namely, the contact angle $\alpha$. By \cite{Bialy2}, we have the following formulae for the first and second partials of $L$:
$$ L_x = - \cos {\alpha}, L_y = \cos{\alpha} $$
\begin{equation}\label{eq:S2Lderiv}
L_{xy}=\frac{\sin^2{\alpha}}{\sin L} \quad L_{xx} = \frac{\sin^2 \alpha}{\tan L}
- \kappa(x) \sin \alpha \quad L_{yy} = \frac{\sin^2 \alpha}{\tan L} - \kappa(y) \sin \alpha.
\end{equation}
Once again, we seek a parameterization on the curve such that  the values of the parameter at points $x$ and $y$ differ by a constant: $t(y)=t(x)+2c$.
\begin{proposition} \label{newparam}
The desired parameterization $\gamma(t)$ is given by the equation 
\begin{equation*}
x' =\frac{a}{\sqrt{\kappa^2(x)+\sin^2{\alpha}}}
\end{equation*}
 where $a$ is a constant.
\end{proposition}
\begin{proof}
 Equation (\ref{eq:Lderiv_simple}) holds along our curve as before, so $L_{xx} L_{yy} = L_{xy}^2$. Substitute from \ref{eq:S2Lderiv} to obtain the equation
\begin{equation} \label{roots}
\left(\kappa(x) -\frac{\sin\alpha}{\tan L}\right)\left(\kappa(y) -\frac{\sin\alpha}{\tan L}\right)=\frac{\sin^2 \alpha}{\sin^2 L}.
\end{equation} 
Then we can compute  $y'/x'$ from equation (\ref{eq:Lderiv_simple}):
\begin{equation} \label{ratio}
\frac{y'}{x'}=-\frac{-L_{xx}}{L_{xy}}=\left(\kappa(x)-\frac{\sin{\alpha}}{\tan L}\right)\frac{\sin L}{\sin{\alpha}}
=\frac{\sqrt{\kappa(x) - \frac{\sin\alpha}{\tan L}}}{\sqrt{\kappa(y) - \frac{\sin\alpha}{\tan L}}},
\end{equation}
the last equality due to (\ref{roots}).
Next, we claim that
\begin{equation} \label{root2}
\frac{\sqrt{\kappa(x) - \frac{\sin\alpha}{\tan L}}}{\sqrt{\kappa(y) - \frac{\sin\alpha}{\tan L}}} =
\frac{\sqrt{\kappa^2(x)+\sin^2{\alpha}}}{\sqrt{\kappa^2(y)+\sin^2{\alpha}}},
\end{equation}
which, along with (\ref{ratio}), implies the statement of the proposition.

It remains to prove (\ref{root2}). Rewrite (\ref{roots}) as 
$$
\kappa(x)\kappa(y)-\frac{\sin\alpha}{\tan L} (\kappa(x)+\kappa(y)) -\sin^2\alpha=0,
$$
and multiply by $\kappa(y)-\kappa(x)$ to obtain
$$
\kappa(x)\kappa^2(y)-\frac{\sin\alpha}{\tan L}\kappa^2(y)+\kappa(x)\sin^2\alpha 
= \kappa^2(x)\kappa(y)-\frac{\sin\alpha}{\tan L}\kappa^2(x)+\kappa(y)\sin^2\alpha,
$$
or
$$
\left(\kappa(x) - \frac{\sin\alpha}{\tan L}\right) (\kappa^2(y)+\sin^2{\alpha}) =
\left(\kappa(y) - \frac{\sin\alpha}{\tan L}\right) (\kappa^2(x)+\sin^2{\alpha}). 
$$
This implies (\ref{root2}).
\end{proof}

We  choose $a$ in such a way that 
\begin{equation} \label{choice}
T=\frac{1}{a}\int_0^{L(\gamma)}
\sqrt{\kappa^2(x)+\sin^2{\alpha}}~ \mathrm{d}x=2\pi,
\end{equation}
 in order to make Fourier expansion more convenient. 
 
Define a function $f$ on the curve by  
\begin{equation} \label{newf}
\cot f = \frac{\kappa}{\sin {\alpha}}.
\end{equation}

\begin{remark}
The meaning of the function $f$ is illustrated in Figure \ref{circle}. Let $O$ be the center of the osculating circle at point $x\in\gamma$, and let $R$ be its radius. Then $\cot R = \kappa(x)$. Drop the perpendicular from $O$ to the segment $xy$. Then we have a right triangle $PxO$ with an angle $\pi/2 -\alpha$. Solving a right spherical trianle yields $\cot |Px| \sin \alpha = \cot R$. Hence $f=|Px|$.
\end{remark}

\begin{figure}[H]
\centering
\includegraphics[height=2.5in]{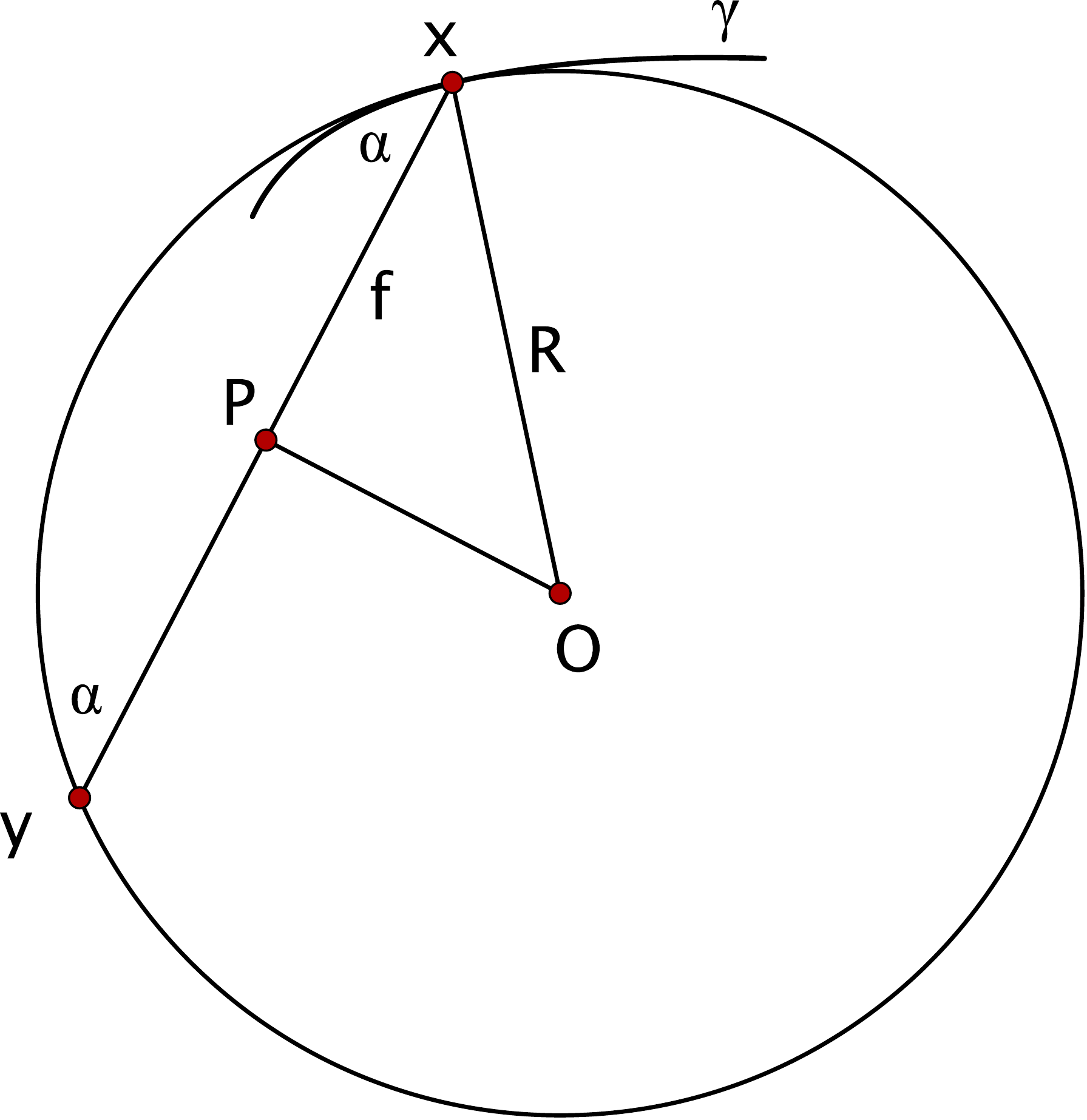}
\caption{\label{circle} Geometric interpretation of the function $f$}
\end{figure}

Denote by $f_1$ and $f_2$ the values of this function at points $x$ and $y$.

\begin{proposition} \label{mainprop}
One has
\begin{equation} \label{eq:S2f}
{a\cot{\alpha}}\ (\sin f_1-\sin f_2)={f'_1}+{f'_2}.
\end{equation} 
\end{proposition} 
\begin{proof}
First, note that Proposition \ref{newparam} and (\ref{newf}) imply that
\begin{equation} \label{newx'}
x'=\frac{a \sin f}{\sin\alpha}.
\end{equation}
Next, as before, $L_{xx}L_{yy}=L^2_{xy}$, and substituting from equation (\ref{eq:S2Lderiv}), we obtain
\begin{equation*} 
\cot L=\frac{\kappa(x)\kappa(y)-\sin^2{\alpha}}{\kappa(x)\sin
\alpha+\kappa(y)\sin \alpha}.
\end{equation*}
Substituting $\kappa(x)$ and $\kappa(y)$ from (\ref{newf}) yields 
$$
\cot L=\frac{\cot f_1 \cot f_2 -1}{\cot f_1+\cot f_2}=\cot(f_1+f_2).
$$
Thus, $L=f_1+f_2$, and hence ${L'}={f'_1}+{f'_2}$. By the chain rule,  
$$
L'= L_x {x'}+L_y {y'}=\frac{a\cos{\alpha}}{\sin{\alpha}}\big(\sin f_1-\sin
f_2),
$$
where the last equality is due to (\ref{eq:S2Lderiv}) and (\ref{newx'}). This implies the statement.
\end{proof}

\begin{remark}
{\rm
Equation (\ref{eq:S2f})  appeared in \cite{Tabachnikov2}, in a study of a different rigidity problem, also related to a flotation problem (Ulam's Problem on bodies that float in equilibrium in all positions), and to a problem of bicycle kinematics.
}
\end{remark}

Equation (\ref{eq:S2f}) 
is an analog of equation (\ref{eq:E2f}) but, unlike the Euclidean case, it is non-linear, and we do not know how to solve it. Thus we resort to linearization of the problem, that is, start from a circle $\gamma_0$ of radius $R$  and then deform it to find infinitesimal solutions.
 
Write 
$
f_1(t)=f(t+c),\ f_2(t)=f(t-c),
$
where the constant $c$ depends on the Gutkin curve and the contact angle (in the Euclidean case, $c=\alpha$).
For a circle on $S^2$, we compute the relation between $R, \alpha$ and $c$, and the value of $a$.

\begin{lemma} \label{canda}
One has
\begin{equation*}
\cos\alpha=\frac{\cos c}{\sqrt{\sin^2{R}\cos^2{c}+\cos^2{R}}},\ {\rm or\ equivalently,}\ \cot c = \cos R \cot \alpha,
\end{equation*}
and
\begin{equation*}
\ a=\sqrt{\cos^2 R + \sin^2 \alpha \sin^2 R}.
\end{equation*}
\end{lemma}
\begin{proof}
The circle of radius $R$ is parameterized as 
$$
\gamma_0(t)=(\sin R \cos t,\sin R \sin t,\cos R),
$$
where $t\in[0,2\pi]$. We need to find the angle $\alpha$ made by the geodesic segment $[\gamma_0(-c),\gamma_0(c)]$ with this circle.

The great circle through points $\gamma_0(-c)$ and $\gamma_0(c)$ is the parametric curve
$$
\Gamma(s)=\frac{\cos c}{\sqrt{\sin^2{R}\cos^2{c}+\cos^2{R}}} (\sin R \cos c, 0, \cos R) + \sin s (0,1,0),
$$
and $\Gamma(s_0)=\gamma_0(c)$ for  $\sin s_0 = \sin R \sin c$. It remains to compute the velocity vectors $d\Gamma(s)/ds$ and $d\gamma_0(t)/dt$, evaluate them at $s=s_0$ and $t=c$, respectively, and compute the angle between these vectors. This straightforward computation yields the first formula of the lemma. A calculation using trigonometric identities yields the simpler, equivalent, formula.

To obtain the formula for $a$, note that the length and the geodesic curvature of the circle $\gamma_0$ are equal to $2\pi\sin R$ and $\cot R$, respectively. Then (\ref{choice}) yields the result.
\end{proof}

Now we are ready for the proof of Theorem \ref{Conscurv} in the spherical case. Let $\gamma_0$ be a circle of radius $R$. Then the function $f$ is a constant  satisfying $\cot f=\cot R/\sin\alpha$, see (\ref{newf}), and the constants $c$ and $a$ are as in Lemma \ref{canda}. Consider an infinitesimal deformation of the curve in the class of Gutkin curves with the contact angle $\alpha$. Then $f,c$ and $a$ deform as follows
$$
f\mapsto f+\varepsilon g(t),\ c\mapsto c+\varepsilon \delta,\ a\mapsto a+\varepsilon \beta,
$$
where $g(t)$ is a $2\pi$-periodic function, and all the previos relations hold. 
Substitute into equation (\ref{eq:S2f}):
\begin{equation*}
\begin{split}
(a+\varepsilon \beta) \cot \alpha\ (\sin (f+\varepsilon g(t+c+\varepsilon\delta))-\sin (f+\varepsilon g(t-c-\varepsilon\delta)))=\\\varepsilon(g'(t+c+\varepsilon\delta)+g'(t-c-\varepsilon\delta)).
\end{split}
\end{equation*}
Computing modulo $\varepsilon^2$ yields
$$
a \cot\alpha \cos f\ (g(t+c)-g(t-c)) = g'(t+c)+g'(t-c).
$$
As before, this implies that if $g(t)$ is not a constant (which would correspond to a trivial deformation to a circle of possibly different radius) then 
$$
k \cos kc = a \cot\alpha \cos f \sin kc
$$
for each $k$ for which the Fourier coefficient $b_k\ne 0$. Substituting the values of the constants $f$ and  $a$ and eliminating $\alpha$ using Lemma \ref{canda} yields, after a straightforward, albeit tedious, computation:
$$
k \cos kc = \cot c \sin kc\quad {\rm or}\quad k \tan c = \tan kc.
$$

For $k=1$, this formula holds for all $c$, and it remains to explain the condition $k\ge 2$ in the formulation of the theorem. The next proposition shows that the first Fourier coefficient $b_1$ vanishes.

\begin{proposition}
The function $g(t)$ is $L^2$-orthogonal to the first harmonics, that is, its Fourier expansion does not contain $\cos t$ and $\sin t$.
\end{proposition}

\begin{proof}
Let $\varphi,\theta$ be the spherical coordinates. Recall that the spherical metric is 
$
\sin^2\theta\ d\phi^2 + d\theta^2.
$
The non-perturbed curve is $\gamma_0(t) = (t,R)$, the circle of latitude of radius $R$. Consider its infinitesimal deformation 
$$\gamma_{\varepsilon}(t)=(t+\varepsilon f(t), R+\varepsilon g(t)),$$ 
where $f$ and $g$ are $2\pi$-periodic functions. The curvature of $\gamma_0$ is $\cot R$. Let $\cot R + \varepsilon k(t)$ be the curvature if $\gamma_{\varepsilon}$. Here and below, all computations are modulo $\varepsilon^2$.

Due to (\ref{newf}), 
$$\sin\alpha \cot (f+\varepsilon g(t)) = \cot R + \varepsilon k(t),$$
 hence, up to a constant multiplier, $g=k$. We shall compute $k(t)$ and show that it is free from first harmonics. 
 
We shall use Liouville's formula for curvature of a curve in an orthogonal coordinate system $(u,v)$, see, e.g., \cite{Ca}. Recall this formula. Let $\psi$ be the angle made by the curve with the curves $v=const$, let $K_u$ and $K_v$ be the geodesic curvatures of the coordinate curves $v=const$ and $u=const$, and let $x$ be the arc length parameter on the curve. Then the curvature of the curve is
\begin{equation} \label{Liou}
\frac{d\psi}{dx} + K_u \cos\psi + K_v \sin\psi.
\end{equation}
In our situation, $u$ and $v$ are the longitude and latitude, so $K_v=0$ and $K_u(\varphi,\theta)=\cot\theta$.

Since
$$
x=\sin\theta \cos\varphi,\ y=\sin\theta \sin\varphi,\ z=\cos\theta,
$$
one has: $\gamma_{\varepsilon} =$
\begin{equation*}
\begin{split}
(\sin R \cos t +\varepsilon (g(t)\cos R \cos t - f(t) \sin R \sin t),\\
\sin R \sin t +\varepsilon (g(t)\cos R \sin t + f(t) \sin R \cos t), \\
\cos R -\varepsilon g(t) \sin R).
\end{split}
\end{equation*}
Then $\gamma'_{\varepsilon} =$
\begin{equation*}
\begin{split}
(-\sin R \sin t +\varepsilon (-g\cos R \sin t +g' \cos R \cos t - f \sin R \cos t -f' \sin R \sin t),\\
\sin R \cos t +\varepsilon (g\cos R \cos t + g' \cos R \sin t -f \sin R \sin t + f' \sin R \cos t), \\
 -\varepsilon g' \sin R).
\end{split}
\end{equation*}
It follows that 
$$
|\gamma'_{\varepsilon}|= \sin R + \varepsilon (g \cos R + f' \sin R).
$$
The angle $\psi$ between $\gamma'_{\varepsilon}$ and the circles of latitude is infinitesimal. Therefore $\cos\psi=1$ (modulo $\varepsilon^2$).
Using the formula for $\gamma'_{\varepsilon}$, one  
computes this angle:
$$
\psi=-\varepsilon \frac{g'(t)}{\sin R}
$$
(minus sign is due to the fact that increasing $g$ pushes the curve down to the equator). Hence
$$
\frac{d\psi}{dx}=\frac{\psi'}{x'}= \frac{\psi'}{|\gamma'_{\varepsilon}|}= -\varepsilon \frac{g''(t)}{\sin^2 R}.
$$
Finally, 
$$
\cot \theta = \cot (R+\varepsilon g(t))=\cot R - \varepsilon \frac{g(t)}{\sin^2 R}.
$$
Now (\ref{Liou}) implies that, up to a constant factor, $k(t)=g(t)+g''(t)$. Since the differential operator $d^2/dx^2 + 1$ ``kills" the first harmonics,  the result follows.
\end{proof}

This concludes the proof in the spherical case.

For the case of $\mathbb{H}^2$, we apply a similar method, so we briefly describe the differences. The  
formulas for the partials of $L$ read
\cite{Bialy2}:
$$ L_x = - \cos {\alpha}, L_y = \cos{\alpha}, $$
\begin{equation*}
L_{xy}=\frac{\sin^2{\alpha}}{\sinh L}, \quad L_{xx} = \frac{\sin^2 \alpha}{\tanh
L} - \kappa(x) \sin \alpha, \quad L_{yy} = \frac{\sin^2 \alpha}{\tanh L} -
\kappa(y) \sin \alpha.
\end{equation*}

The parameterization of a Gutkin curve is given by the formula $x_t = a/\sqrt{\kappa(x)^2-\sin^2{\alpha}}$ where the constant $a$ is normalized so that the parameter $t$ takes values in $[0,2\pi]$. One defines the function $f(t)$ by $\coth f = \kappa/\sin\alpha$, and as before, one obtains a difference-differential equation
\begin{equation*}
{a\cot{\alpha}}\ \big(\sinh f_1-\sinh f_2)={f'_1}+{f'_2}.
\end{equation*} 
Analogs of Lemma \ref{canda} hold:
\begin{equation*}
\cos\alpha=\frac{\cos c}{\sqrt{\cosh^2{R}-\sinh^2{R}\cos^2{c}}},\ {\rm or\ equivalently,}\ \cot c = \cosh R \cot \alpha,
\end{equation*}
and
\begin{equation*}
a=\sqrt{\cosh^2 R - \sin^2 \alpha \sinh^2 R}.
\end{equation*}

The computations in Euclidean space $\mathbb R^3$ involving the unit sphere are replaced by similar computations in the Minkowski space $\mathbb R^{1,2}$ involving hyperboloid of two sheets, used as a model of $\mathbb{H}^2$.

\section{Gutkin Polygons} \label{sect:polygon}

Refer to the introduction for the definition of a Gutkin $(n,k)$-gon. Let $G(n,k)$ denote the set of all Gutkin $(n,k)$-gons. Given $P \in G(n,k)$,  it will be convenient to think of $P$ as being embedded in the complex plane $\mathbb{C}$. Let $l_i$ denote the side length, $|v_{i+1} -v_i|$. 

Notice that if $n = 2k$, for every index, $i$, one has $i - k = i + k$. Therefore in this case, each vertex is the end point of exactly one diagonal. 
If $n \neq 2k$ then $i -k \neq i+k$, so each vertex is the endpoint of two diagonals. In this case, for each
$v_i$, we call the angle between the two diagonals $\beta_i$, i.e. $\beta_i =
\angle v_{i-k} v_i v_{i+k}$. 

\begin{figure}
 \centering
        \begin{subfigure}[b]{0.46\textwidth}
                \includegraphics[width=\textwidth]{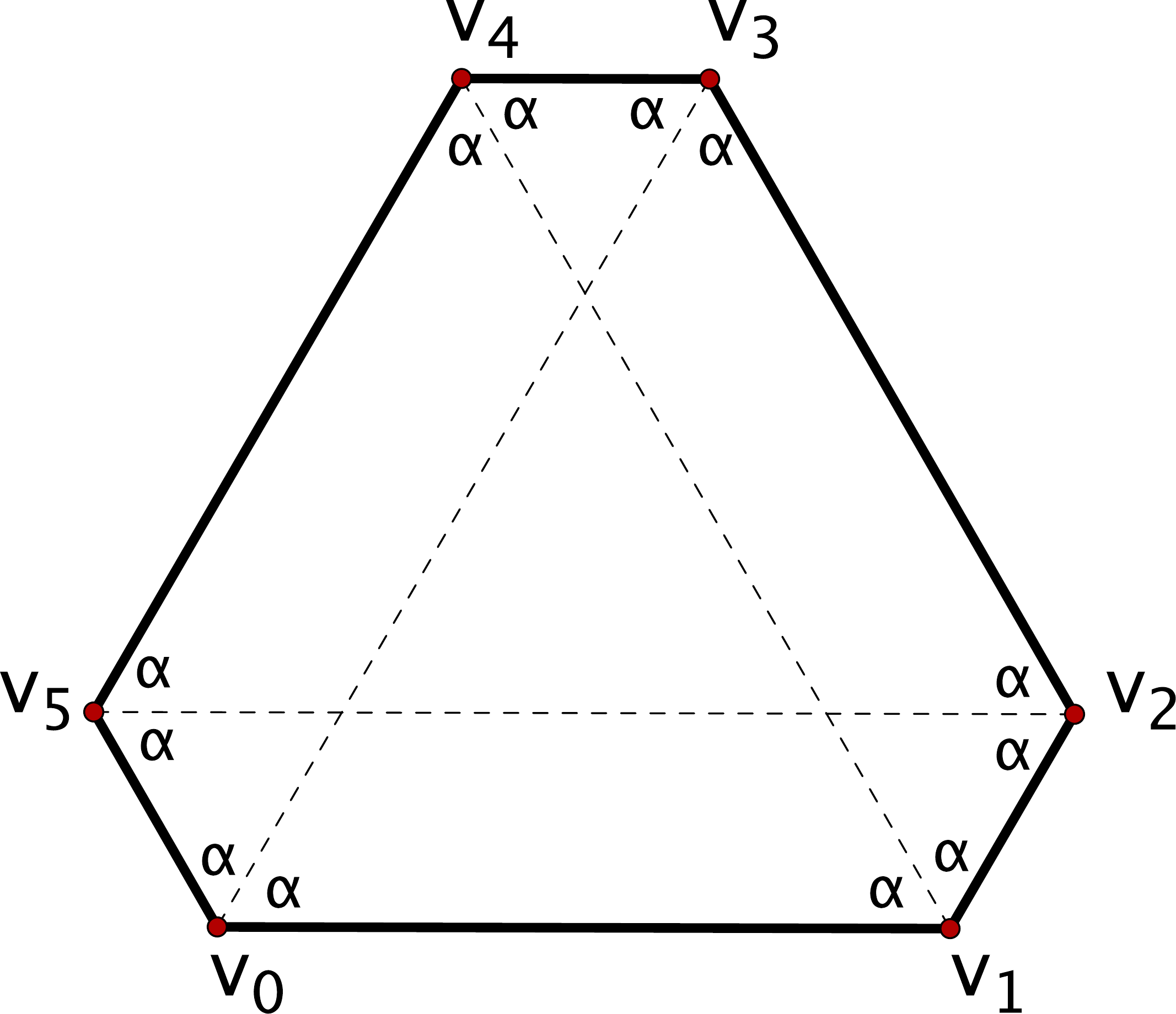}
                \caption{Gutkin $(6,3)$-gon.}
                \label{fig:6-3-withalphas}
        \end{subfigure} \quad
        \begin{subfigure}[b]{0.46\textwidth}
                \includegraphics[width=\textwidth]{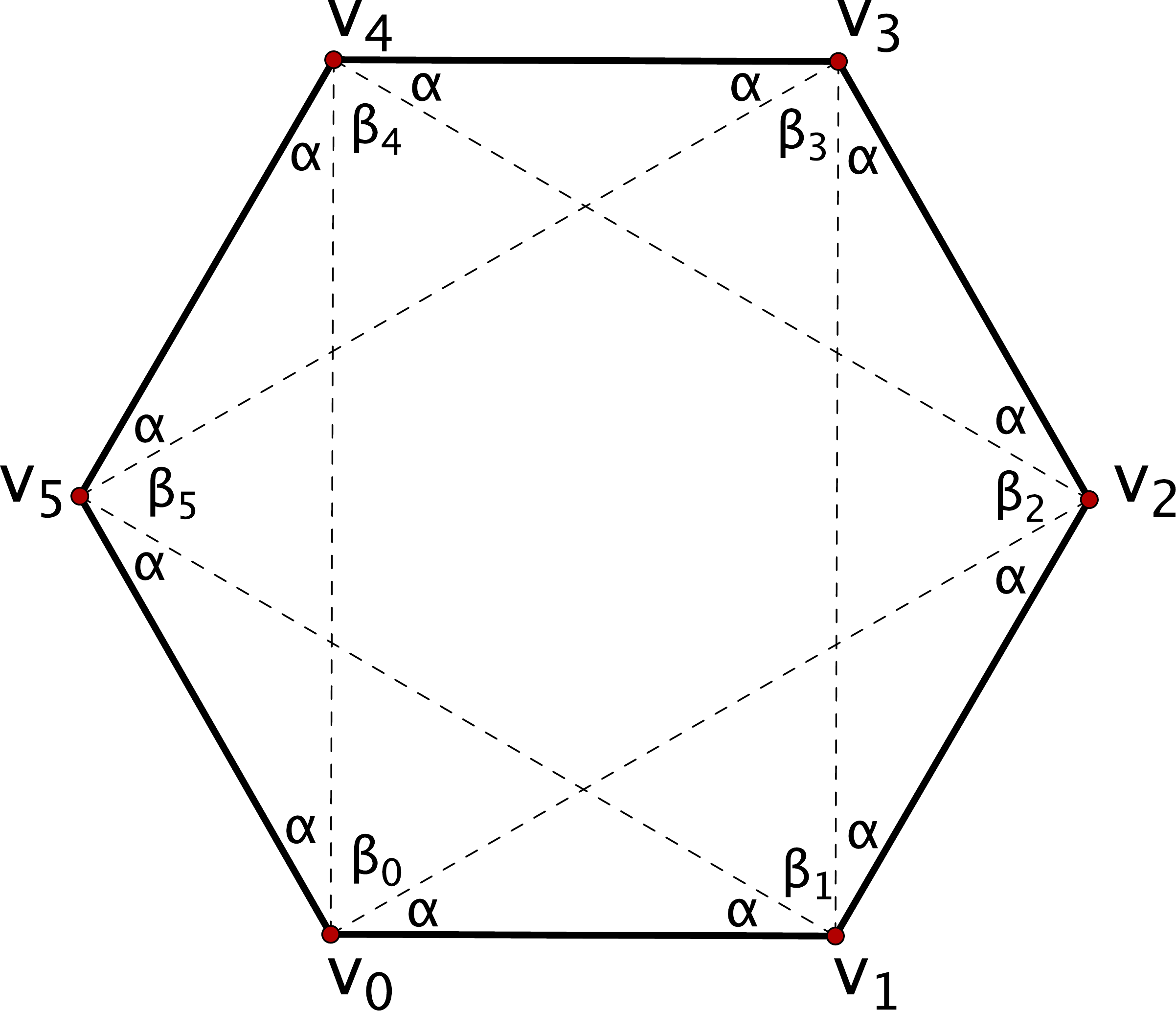}
                \caption{Gutkin $(6,2)$-gon.}
                \label{fig:6-2-withbetas}
        \end{subfigure}
        \caption{Two Gutkin polygons with angles labelled.}\label{fig:angles}
\end{figure}

The first two propositions in this section will establish basic geometric properties of a Gutkin $(n,k)$-gon. 

\begin{proposition}
Given $n$ and $k$, the associated contact angle is equal to $\pi(k-1)/n$ for any Gutkin $(n,k)$-gon.
\end{proposition}

\begin{proof}
Let $P \in G(2k, k)$ for some $k \geq 2$. For each $i$, $\angle v_{i+k} v_i v_{i+1} = \angle v_{i+k} v_i v_{i-1}$ = $\alpha$. Then all interior angles of $P$ are equal to $2 \alpha$. Since the sum of the interior angles of any $n$-gon is equal to $\pi(n-2)$, we have $\alpha = \pi (n-2)/(2n)$, which is
equal to $\pi (k-1)/n$.

When $n \neq 2k$, it suffices to show that $\alpha$ is determined by $n$ and $k$. First, note that the sum of the interior angles of the Gutkin polygon equals
$(n-2) \pi$ and also equals 
$$\sum_{i = 0}^{n-1} \beta_i + 2n \alpha,$$
see Figure \ref{fig:angles}. Therefore, 
$$\alpha = \frac{\pi (n-2) - \sum_{i=0}^{n-1} \beta_i}{2n}.$$
 To show that $\alpha$ is determined by $n$ and $k$, we show that $\sum_{i=0}^{n-1} \beta_i$ is determined by $n$ and $k$.

For fixed $n$ and $k$, let $P\in G(n,k)$. For $1 \leq j \leq \gcd(n,k-1)$, define the polygon 
$$Q_j = \overline{v_j v_{j+k} v_{j+2k} \cdots v_{j + (nk / \gcd(n,k-1))-1}}.$$
 Two examples of $Q_j$'s are shown in Figure \ref{fig:Q0}. Note that the sides of $Q_j$ are the diagonals of $P$. The
vertices of all $Q_j$'s form a disjoint partition of $\{v_0,v_1,\dotsc, v_{n-1} \}$ into $\gcd(n,k-1)$  subsets of equal size. Thus, the sum of the interior angles of all $Q_j$'s are $\sum_{i=0}^{n-1} \beta_i$. Since the sum of the interior angles of $Q_j$ is $\pi(n/\gcd(n,k-1) -2)$ for all $j$, $\sum_{i=0}^{n-1} \beta_i$ is determined by $n$ and $k$.

\begin{figure}
        \centering
        \begin{subfigure}[b]{0.46\textwidth}
                \includegraphics[width=\textwidth]{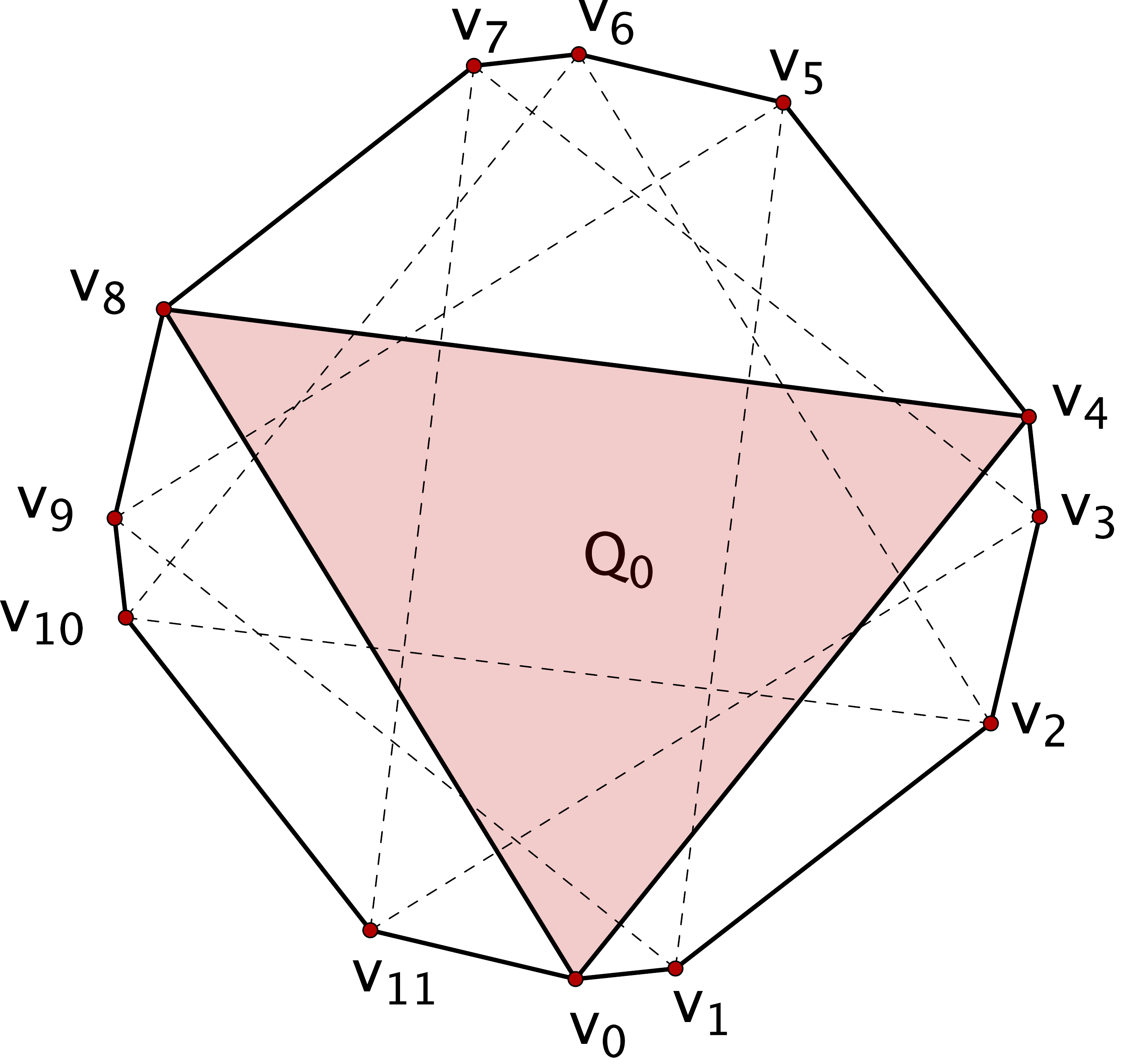}
                \caption{$Q_0$ for a Gutkin $(12,4)$-gon}
                \label{fig:12-4-withQ0}
        \end{subfigure} \quad
        \begin{subfigure}[b]{0.46\textwidth}
                \includegraphics[width=\textwidth]{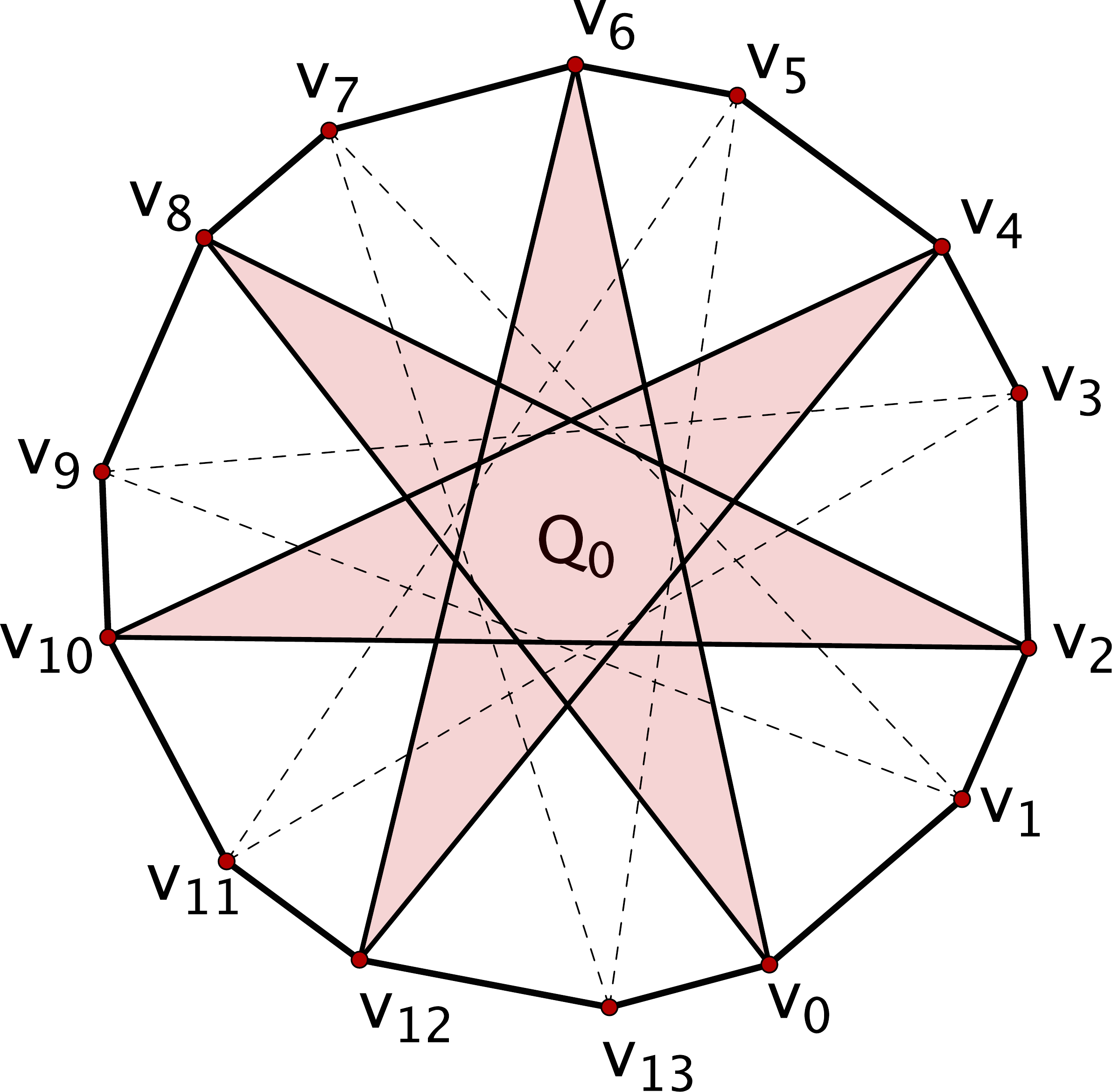}
                \caption{$Q_0$ for a Gutkin $(14,6)$-gon}
                \label{fig:14-6-withQ0}
        \end{subfigure}
        \caption{Polygons $Q_0$ on two Gutkin polygons.}\label{fig:Q0}
\end{figure}

For a regular polygon, $\alpha = \pi (k-1) / n $. Since $\alpha$ is determined by $n$ and $k$, the above equation is true for all polygons in $G(n,k)$.
\end{proof}

\begin{proposition} \label{prop:angle}
In a Gutkin $(n,k)$-gon, the interior angles associated to vertices $v_{i}$ and $v_{i+k-1}$ are equal for all $i$.
\end{proposition}

\begin{proof} 
Consider the self-intersecting quadrilateral $B_i = v_i v_{i+k} v_{i+k+1} v_{i+1}$, see Figure \ref{fig:12-4-B4}. Let $w_{i}$ denote the intersection point of the two diagonals, $\overline{v_i v_{i+k}}$ and $\overline{v_{i+1} v_{i+k+1}}$. Notice that $B_i$ is comprised of two triangles meeting at $w_i$.
The opposite angles at $w_i$ are equal, and the angle at $v_i$ and $v_{i+k+1}$ is equal to $\alpha$. Therefore the angles at $v_{i+1}$ and $v_{i+k}$ are equal, which are also equal to $\alpha + \beta_{i+1}$ and $\alpha + \beta_{i+k}$, respectively. Then $\beta_{i+1} = \beta_{i+k}$. Since $i$  the interior angle associated to any $v_j$ is equal to $2\alpha + \beta_j$,  the desired result follows.
\end{proof}

\begin{figure}
  \centering
    \includegraphics[width=.5\textwidth]{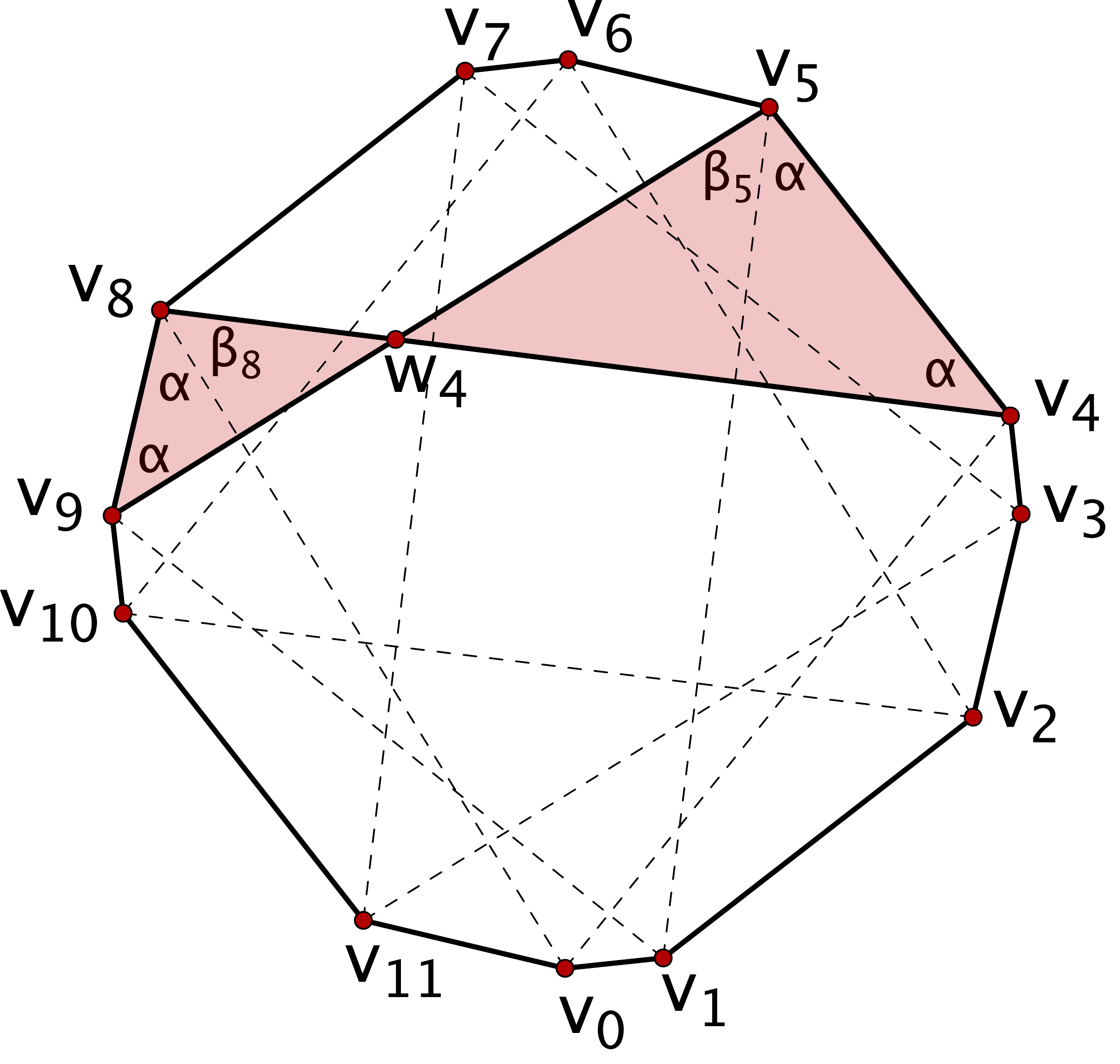}
  \caption{A Gutkin $(12,4)$-gon. The shaded region is $B_4$.}
  \label{fig:12-4-B4}
\end{figure}

\begin{corollary} \label{cor:angle}
If $n$ and $k-1$ are co-prime, then any $P\in G(n,k)$ is equiangular. 
\end{corollary}

The case $n=2k$ is special in that Gutkin polygons abound (in the continuous case, this corresponds to the contact angle $\pi/2$, that is, when Gutkin curves are curves of constant width). Let $\mathbb{R}^{n}_{+}$ be the positive ortant. 

\begin{proposition}
The dimension of the space of Gutkin $(2k,k)$-gons, considered modulo similarities, equals $k-2$. This quotient space is 
the intersection of a $(k-2)$-dimensional affine subspace with $\mathbb{R}_+^{k}$. 
\end{proposition}

\begin{proof}
Let $P$ be a Gutkin $(2k,k)$-gon. Consider the diagonals $\overline{v_i v_{i+k}}$ and $\overline{v_{i+1} v_{i+k+1}}$ of $G(2k,k)$, see  Figure \ref{fig:6-3-B1}. Let $w_{i}$ denote the intersection of these two diagonals, and let $B_i$ be the bow-tie-shaped polygon, $\overline{v_i v_{i+1} v_{i+k+1} v_{i+k}}$. Notice that $\triangle v_i v_{i+1} w_i$ and $\triangle v_{i+k+1}
v_{i+k} w_i$ are both isosceles triangles and are similar.

Thus, $v_i w_i = v_{i+1} w_i$ and $v_{i+k} w _i = v_{i+k+1} w_i$. Hence, the
diagonals $v_i v_{i+k}$ and $v_{i+1} v_{i+k+1}$ have equal length. Since $i$ is arbitrary and the indices are circular, all diagonals have the same length, say, $h$. Since $h$ is just a scaling factor, we set $h=1$ for the remainder of the proof.

\begin{figure}
  \centering
    \includegraphics[width=.5\textwidth]{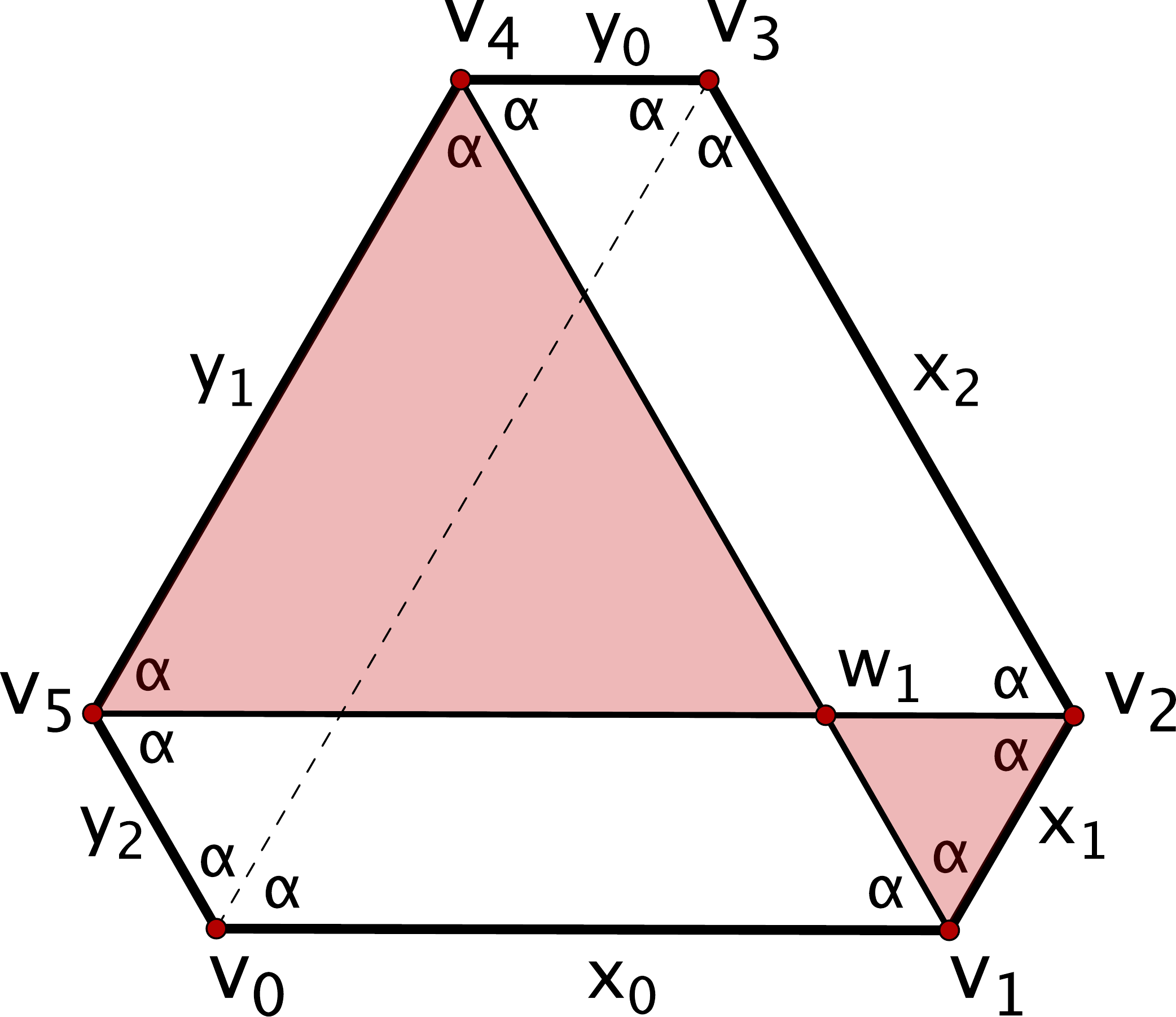}
  \caption{A Gutkin $(6,3)$-gon with side lengths labeled. The shaded region is
  $B_1$.}
  \label{fig:6-3-B1}
\end{figure}

 Notice that $P$ is comprised of $k$ polygons $B_i$'s. Let $x_i$ denote the length of $\overline{v_i v_{i+1}}$ for $0 \leq i \leq k-1$, and let $y_i$ denote the length of
$\overline{v_{i+k} v_{i+k+1}}$, where $0 \leq i \leq k-1$. Note that $x_i$ and $y_i$ denote length of the non-intersecting sides of $B_i$.

Assume that $v_{0}$ is at the origin and $v_{1}$ lies on the positive $x$ axis, and recall that the vertices are labeled in counter-clockwise order. This factors out the action of the isometry group of the plane. 
We shall show that $x_0,\ldots,x_{k-1}$ uniquely determine $y_0,\ldots,y_{k-1}$ and study the condition that these sides  form a closed polygon.

Since the diagonals have fixed length equal to $1$, one has $y_i = 2 \cos \alpha - x_i$. Also, $v_k$ is at the point $(\cos \alpha,  \sin \alpha)$.  Viewing the sides of $G(2k,k)$ as vectors, the $i^{th}$ side is  $x_i(\cos i \theta, \sin i \theta)$, where $\theta = \pi - 2 \alpha = \pi /k$, and the sum of these vectors must be equal to $v_k$. Thus
\begin{equation} \label{xcond}
\sum_{i=0}^{k-1} x_i(\cos i \theta, \sin i \theta) = ( \cos \alpha, \sin
\alpha)
\end{equation}
If the side lengths $x_0,\dotsc,x_{k-1},y_0, \dotsc, y_{k-1}$ form a closed polygon then the sides with lengths $y_i$ must start at $v_k$ and end at $v_0$. In other words, the side lengths satisfy
\begin{equation} \label{ycond}
v_k + \sum_{i=0}^{k-1} y_i(\cos (\pi + i \theta), \sin ( \pi + i \theta)) = v_0.
\end{equation}
Simplifying the left hand side yields
\begin{align*}
(\cos \alpha, \sin \alpha) +& \sum_{i=0}^{k-1} y_i (- \cos i \theta,  - \sin i \theta) =\\
 (\cos \alpha, \sin \alpha) -& \sum_{i=0}^{k-1} (2 \cos \alpha - x_i) (\cos i
\theta, \sin i \theta)=\\
 (\cos \alpha, \sin \alpha) -&  2 \cos \alpha \sum_{i=0}^{k-1} (\cos i
\theta, \sin i \theta) +   \sum_{i=0}^{k-1} x_i (\cos i
\theta, \sin i \theta) =\\
 (\cos \alpha, \sin \alpha) -& 2 \cos \alpha (1, \tan \alpha) +  
(\cos \alpha, \sin \alpha)
= (0,0) = v_0.
\end{align*}
Thus (\ref{xcond}) implies (\ref{ycond}).

Hence, $G(2k,k)$ is determined by the $k$-tuple $x_0,\ldots,x_{k-1}$ satisfying 
the  two linear equations (\ref{xcond}). This concludes the proof. 
\end{proof}

 Next we consider other equiangular cases.

\begin{proposition} \label{thm:discretetangent}
The quotient space of the space of equiangular Gutkin $(n,k)$-gons by the group of similarities is identified with  
the intersection of an $M$-dimensional affine subspace with  $\mathbb{R}^n_+$, where $M$ is equal to the number of positive integers $2 \le r \le n-2$ satisfying the equation
\begin{equation} \label{restr2}
\tan \left(\frac{kr \pi}{n}\right) \tan \left(\frac{\pi}{n}\right) = \tan \left(\frac{k \pi}{n}\right) \tan \left(\frac{r \pi}{n}\right). 
\end{equation}
\end{proposition}

\begin{proof}
Let $P \in G(n,k)$ be embedded in the complex plane, with $v_0 = 0$ and $v_1$ on the positive real axis. Let $x_i = |v_{i+1} - v_i|$ for $0 \leq i \leq k-1$ be the side lenths of $P$. Let $\omega = \exp ( 2\pi/n)$. Notice that $v_{i+1} - v_i = x_i \omega^i$, and a diagonal can be
represented as 
\begin{equation} \label{eq:diagonal0} v_{i+k} - v_i = a_i \omega ^{i+m}
\end{equation} where $a_i \in \mathbb{R}$,
$a_i > 0$, and $m = (k-1)/2$. Notice that in this representation, 
\begin{align*}
\arg(v_{i+1} - v_i) &= (2 \pi i)/n, \\ 
\arg(v_{i+k} - v_{i+k-1}) &= 2,
\pi(i+k-1)/n \\ \arg(v_{i+k} - v_i) &= \pi (2i + k -1)/n.
\end{align*}
Then 
$$\angle v_{i+1} v_i v_{i+k} = \angle v_{i+k-1} v_{i+k} v_i = \pi(k-1)/n = \alpha.$$
 Moreover,
\begin{align*}
v_{i+k} - v_i &= (v_{i+k} - v_{i+k-1}) + (v_{i+k-1} - v_{i+k-2})  + \dotsb +
(v_{i+1} - v_i) \\
&= \omega^{i+k-1} x_{i+k-1}  + \omega^{i+k-2} x_{i+k-2} + \dotsb + \omega^i
x_i\\
&= \omega^i x_i + \omega^{i+1} x_{i+1} \ + \dotsb + \omega^{i+k-1} x_{i+k-1}.
\end{align*}

From  (\ref{eq:diagonal0}), $v_{i+k} - v_i$ is also equal
to $a_i \omega ^{i+m}$. Thus
\begin{align*}
a_i \omega ^{i+m} &= \omega^i x_i + \omega^{i+1} x_{i+1} \ + \dotsb +
\omega^{i+k-1} x_{i+k-1}\\
a_i &= \omega^{-m} x_i + \omega^{1-m} x_{i+1} + \dotsb + \omega^{k-1-m} x_{k-1}.
\end{align*}
Using $a_i - \overline{a_i} = 0$, one has
\[
(\omega^{-m} - \omega^m) x_i + (\omega^{1-m} - \omega^{m-1}) x_{i+1} + \dotsb +
(\omega^{k-1-m} - \omega^{m-k+1}) x_{k-1} = 0.
\]
This gives a system of $n$ linear equations on variables $x_i$. The coefficient matrix, $A$, is a circulant matrix where the first row is equal to 
\[
\begin{pmatrix}
\omega^{-m}- \omega^m, & \omega^{1-m}- \omega^{m-1}, & \dotsb, & \omega^{k-1-m}-
\omega^{m-k+1}, &0, &0, &\dotsb, &0
\end{pmatrix}.
\]
Then the eigenvalues of $A$ are 
\begin{equation} \label{eq:eigenvalue1}
\lambda_r = \sum_{\nu=0}^{k-1} (\omega^{\nu-m} - \omega^{m-\nu}) \omega^{\nu r},
\end{equation}
see \cite{Da}.

We expect one of the eigenvalues to be equal to zero because we have not factorized by scaling yet. If no other eigenvalue equals zero then only  trivial solutions exist. 
Now, we compute $\lambda_r$ in three cases: $r=0$, $r=1$ or $r=n-1$, and $2\le r \le n-2$. 

For $r=0$, we have 
\[\lambda_0 = \omega^{-m} \sum_{\nu=0}^{k-1} \omega^{\nu} - \omega^m.\] 
Let $h$ be equal to $|\omega^i + \dotsb + \omega^{i+k}|$. By rotational symmetry, $h$ does not vary with $i$. Now evaluating the above equation, 
\begin{align*}
\lambda_0 = h \omega^{-m} \omega^m - h \omega^m \omega^{-m} = 0. 
\end{align*}
Thus for $r =0$, $A$ has eigenvalue $\lambda_0$ equal to zero. 

For all  other $r$, assume that $\lambda_r$ is equal to zero. Set 
($\ref{eq:eigenvalue1}$) to zero and simplify:
\begin{equation} \label{eq:eigenvalue2}
\sum_{\nu=0}^{k+1} \omega^{(r+1) \nu} = \omega^{k-1} \sum_{\nu=0}^{k-1}
\omega^{(r-1) \nu}.
\end{equation}

For  $r = 1$, equation (\ref{eq:eigenvalue2}) can be written as $k
\omega^{k-1} = \sum_{\nu=0}^{k-1} \omega^{2\nu}$. Then $k = |\sum_{\nu=0}^{k-1} \omega^{2\nu}|$. This is true only if the $\omega^{2\nu}$'s are collinear, which is clearly not the case. Thus $\lambda_1\neq 0$, and likewise for $r=n-1$.

For $2\le r \le n-2$, using geometric series, we can rewrite  \ref{eq:eigenvalue2}
as
\begin{equation} \label{eq:eigenvalue3}
\frac{\omega^{k(r+1)} - 1}{\omega^{r+1} - 1} = \omega^{k-1}
\frac{\omega^{k(r-1)} - 1}{\omega^{r-1} - 1}.
\end{equation}
After expanding this equation in terms of sines and cosines and using
trigonometric identities, one rewrites it as (\ref{restr2}). 
For any solution $r$, one obtains $\lambda_r = 0$. This implies the claim.
\end{proof}

We are ready to prove Theorem \ref{mainpol}.

If $n$ and $k-1$ are coprime then a Gutkin polygon is equiangular by Corollary \ref{cor:angle}.  In \cite{connelly}, Connelly and Csik{\'o}s show that a solution to (\ref{restr2}) for integer values $1 < k, r < n/2$ must satisfy $k + r = n/2$ and $n|(k-1)(r-1)$. Since $n$ and $k-1$ are coprime, there are no solutions. Note also that if $r$ is a solution, so is $n-r$. Thus, by the Proposition \ref{thm:discretetangent}, the matrix $A$ has corank $1$ and the Gutkin polygon must be regular.

It remains to construct a non-trivial Gutkin polygon for non-coprime $n$ and $k-1$. Let $p = \gcd(n,k-1)$ and $q=n/p$. Choose angles
$\theta_1, \dotsc, \theta_p$ such that $\theta_1 + \dotsb + \theta_p = 2 \pi / q$. Divide a unit circle into $q$ equal parts, and divide each of these equal arcs into $p$ arcs of lengths $\theta_1, \dotsc, \theta_p$, in this order. One obtains an inscribed $n$-gon.
See Figure \ref{constr} for $n=4, k=3$.

\begin{figure}
\centering
\includegraphics[height=2.2in]{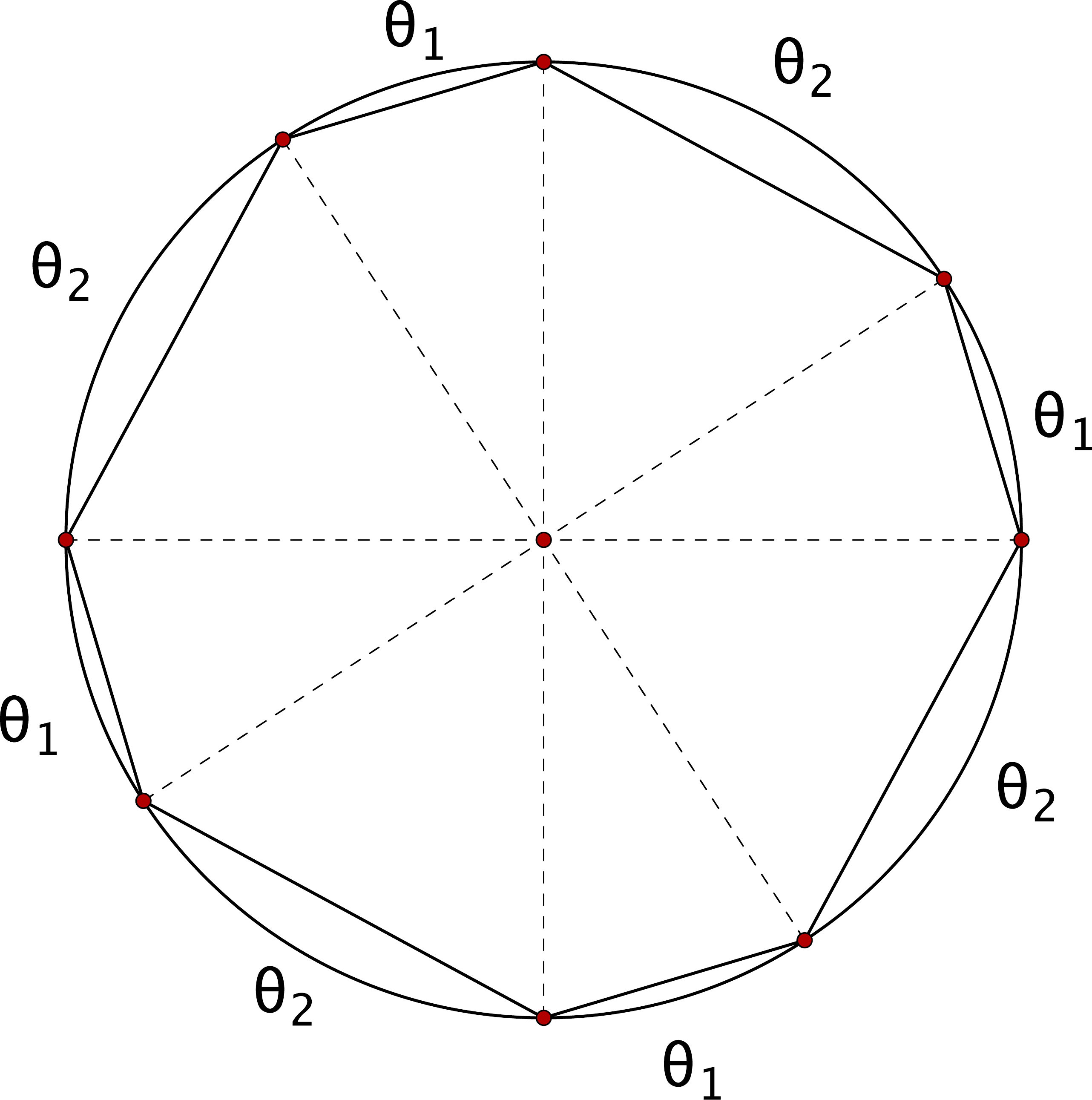}
\caption{\label{constr} Constructing a non-trivial Gutkin polygon}
\end{figure}

\begin{lemma}
The constructed $n$-gon is a Gutkin polygon.
\end{lemma}

\begin{proof}
The angular measure of an inscribed angle is half that of the subtended arc. It follows that 
$$
\angle v_{i+1}v_{i}v_{i+k} =\angle v_{i+k-1}v_{i+k}v_{i} = \frac{\theta_1 + \dotsb + \theta_p}{2} = \frac{\pi}{q}.
$$
\end{proof}

Since the choice of the angles $\theta_1, \dotsc, \theta_p$ was arbitrary, we obtain a $p-1$-parameter family of pairwise non-similar Gutkin polygons.

\bigskip
{\bf Acknowledgments}. This work was done during the  Summer@ICERM 2013 program, and it is a result of collaboration between  undergraduate students and their  advisors. We are grateful to ICERM for its support and hospitality. 
The figures in this paper were made in Cinderellla 2.0. We would like to thank Michael Bialy, Peter J. Lu and Charles Grinstead for interesting discussions.
S. T. was  supported by the NSF grant DMS-1105442.

\end{document}